\documentclass[12pt,reqno,a4paper]{amsart}

\usepackage{amsmath,amssymb,amsthm,amsaddr,mathtools,float}
\usepackage{upgreek,bm,array,enumitem,hyperref}
\usepackage{tikz}
\usetikzlibrary{positioning}

\usepackage[margin=2.5cm,top=3.5cm,bottom=3.5cm,footskip=1cm,headsep=1.5cm]{geometry}
\usepackage[utf8]{inputenc}
\usepackage[british]{babel}
\usepackage{xcolor,graphicx} 
\newcolumntype{M}[1]{>{\centering\arraybackslash}m{#1}}

\newtheorem{lemma}{Lemma}[section]
\newtheorem{theorem}[lemma]{Theorem}

\theoremstyle{definition}
\newtheorem{remark}[lemma]{Remark}
\newtheorem{assumption}[lemma]{Assumption}
\newtheorem{definition}[lemma]{Definition}

\title[Feedback stabilisation for the Cahn-Hilliard equation]%
{On feedback stabilisation for the Cahn-Hilliard equation and its numerical approximation}

\author{%
H.~Egger$^{1,2}$ \and 
M.~Fritz$^{1}$ \and K.~Kunisch$^{1,3}$ \and 
S.~S.~Rodrigues$^{1}$}

\email{herbert.egger@jku.at} 
\email{marvin.fritz@ricam.oeaw.ac.at}
\email{karl.kunisch@uni-graz.at}
\email{sergio.rodrigues@ricam.oeaw.ac.at}

\address{%
\footnotesize
$^1$Johann Radon Institute for Computational and Applied Mathematics, Linz, Austria \\[1pt]
$^2$Institute for Numerical Mathematics, Johannes Kepler University Linz, Austria\\[-0.3pt]
$^3$Institute for Mathematics and Scientific Computing, 
Karl--Franzens University Graz, Austria}

\begin{document}

\begin{abstract}
We consider the stabilisation of solutions to the Cahn-Hilliard equation towards a given trajectory by means of a finite-dimensional static output feedback mechanism. Exponential stabilisation of the controlled state around the target trajectory is proven using careful energy estimates and a spectral condition which characterizes the strength of the feedback. The analysis is general enough to allow for pointwise and distributed measurements and actuation. The main results are derived via arguments that carry over to appropriate discretisation schemes which allows us to establish corresponding exponential stabilisation results also on the discrete level. The validity of our results and the importance of some of our assumptions are illustrated by numerical tests.
\end{abstract}

\maketitle

\begin{quote}
\footnotesize 
\textbf{Keywords:}
Cahn-Hilliard equation, 
exponential stabilisation, 
finite-dimensional feedback,
trajectory tracking,
numerical approximation of controlled systems
\end{quote}

\section{Introduction}
\label{sec:intro}

The Cahn-Hilliard equation is used as one of the fundamental mathematical models for phase separation processes involving conserved quantities, e.g.,
in binary alloys \cite{Cahn1961,CahnHilliard1958,Elliott1989} 
or binary fluids~\cite{Abels2012,Boyer1999,Gurtin96,Hohenberg77}.
Applications include the modeling of polymers \cite{Choksi2009,Zhou2006},
fracture propagation~\cite{Borden2012,Silva2013},
tumor growth~\cite{Fritz2023,Wise2011,Wu2014}, 
image processing~\cite{Bertozzi2007}, 
or topology optimization~\cite{Zhou2007}, 
to mention just a few.
We refer to~\cite{Kim2016,Wu2022,miranville2019cahn} for an overview of basic principles, potential applications, and additional references. 
In this paper, we are interested in the control of solutions to the Cahn-Hilliard equation towards a given target trajectory by feedback control. 

\subsection{Related work}
Using energy estimates, Yong and Zheng~\cite{Yong1991} proved exponential stabilisation of the Cahn-Hilliard equation using a nonlinear distributed feedback mechanism. 
In a seminal work~\cite{AzouaniTiti2014}, Azouani and Titi considered stabilisation of nonlinear dissipative systems, like the one-dimensional Chafee-Infante equation,  by finite-dimensional feedback. Similarly to the results of the present paper, their analysis relies on spectral arguments, energy estimates, and a sufficiently strong feedback operator.
Corresponding computational results were presented by Lunasin and Titi~\cite{LunasinTiti2017}.
Based on solutions of a Riccati operator equation, Barbu et al.~\cite{Barbu2017} established local exponential stabilizability of a nonisothermal Cahn-Hilliard system by a finite-dimensional feedback operator, and in~\cite{AzmiFritzRod24-arx} semiglobal exponential stabilizability was derived using explicit feedback laws with appropriately located actuators.
Gomes~et al.~\cite{Gomes2017}, 
considered the stabilisation of non-trivial solutions of the generalized Kuramoto–Sivashinsky equation using feedback and optimal control techniques. 
Al~Jamal and Morrissiam~\cite{AlJamal2018} studied the linear stability of this equation, and Tomlin and Gomes~\cite{Tomlin2019} investigated the feedback control by point actuators.
Sliding mode control for a phase field system related to tumor growth was studied by Colli et al.~in \cite{Colli2019}. 
Guzman~\cite{Guzman2020} studied the local exact controllability of the 1d Cahn–Hilliard equation to trajectories using Carleman estimates. 
Barros et al.~\cite{Barros2021} used feedback ideas for adaptive time-step control for numerical solvers of a Cahn-Hilliard system with phase-dependent mobility. 
In recent works, Song et al.~\cite{Song2024} and Larios et al.~\cite{Larios2024} investigated data assimilation for the Cahn-Hilliard-Navier-Stokes system and the Kuramato-Sivashinsky equation using feedback ideas. 
A variety of further results exist concerning optimal control for the Cahn-Hilliard equations and related systems. 
Wang~\cite{Wang2011} proposed a semidiscrete algorithm for the distributed control of a Cahn-Hilliard system.
Hinze and Kahle~\cite{Hinze2013} investigated instantaneous control of the Cahn-Hilliard equation, again in a semidiscrete setting.
Zheng and Wang~\cite{Zheng2015} studied optimal control problems for the Cahn-Hilliard equation with state constraints. 
Duan and Zhao~\cite{Duan2015} investigated the optimal control of a viscous Cahn-Hilliard equation.
Hintermüller and co-workers~\cite{Hintermueller2017} studied the optimal control of the Cahn-Hilliard-Navier-Stokes system. 

In this paper, we investigate the exponential stabilizability of solutions to the Cahn-Hilliard equation towards a given target trajectory by means of a finite-dimensional linear static output feedback mechanism. The results are first derived on the continuous level, but they carry over almost verbatim to appropriate discretisations of the problem. 

\subsection{Problem setting}
The system under consideration in this manuscript reads
\begin{alignat}{2}
\partial_t y + \nu \Delta^2 y - \Delta \varphi(y) &= h - \mathcal{F}(y-y_r) \qquad && \text{in } \Omega, \ t>0, \label{eq:ch1} \\
\partial_n y = \partial_n \Delta y &= 0 \qquad && \text{on } \partial\Omega, \ t>0. \label{eq:ch2}
\end{alignat}
Here $y$ is the state, typically a phase fraction, $\nu>0$ a parameter encoding the time and length scales, $\varphi(y)$ the derivative of a double-well potential, responsible for the instabilities of the system, and $\Omega$ the spatial domain. 
Furthermore, $h$ is a given forcing term and $\mathcal{F}$ an appropriate feedback operator. 
The target trajectory $y_r$ is assumed to satisfy 
\begin{alignat}{2}
\partial_t y_r + \nu \Delta^2 y_r - \Delta \varphi(y_r) &= h_r \qquad &&\text{in } \Omega, \ t>0, \label{eq:t1} \\
\partial_n y_r = \partial_n \Delta y_r &= 0 \qquad && \text{on } \partial\Omega, \ t>0, \label{eq:t2}
\end{alignat}
with appropriate right hand side $h_r\in L^2(0,T;L^2({\Omega}))$. 
Note that any smooth function satisfying the boundary conditions \eqref{eq:t2} and, in particular, trajectories or steady states of the uncontrolled system \eqref{eq:ch1}--\eqref{eq:ch2} with $\mathcal{F} \equiv 0$ are suitable target trajectories. 
%

\subsection{Notation and main result}
Let $\Omega \subset \mathbb{R}^d$ be a bounded convex domain and define 
\begin{align} \label{eq:HN2}
H_n^2(\Omega) = \{v \in H^2(\Omega) : \partial_n v=0 \ \text{on } \partial\Omega \}
\qquad \text{and} \qquad 
H_n^{-2}(\Omega) := H_n^2(\Omega)'.
\end{align}
We mostly omit indication of the domain and often write $H_n^{2}$ and $H_n^{-2}$ below. 
We further designate by $L^p(0,T;X)$ and $W^{k,p}(0,T;X)$ the Bochner spaces of functions $u : [0,T] \to X$ with the appropriate integrability, respectively, differentiability in time. Explicit indication of the time interval will again be omitted in many instances. 
Our analysis is based on the following assumptions, which are assumed to hold throughout the text. 
\begin{assumption} \label{ass:1}
$\Omega \subset \mathbb{R}^d$, $d \le 3$ is a bounded and convex domain,  $\nu>0$ constant, $\varphi(y) = y^3 - y$, and $y_r \in L^2(0,T;H^4) \cap H^1(0,T;L^2)$ satisfies \eqref{eq:t1}--\eqref{eq:t2} with $\|y_r\|_{L^\infty(W^{1,\infty})} \le R$. The feedback operator $\mathcal{F} : H_n^2(\Omega) \to H_n^{-2}(\Omega)$ is linear and continuous with $\langle \mathcal{F} z,z\rangle \ge 0$.
\end{assumption}
As will become clear from our analysis, the regularity assumptions on the reference trajectory $y_r$, the special form of the nonlinear term $\varphi(y)$, as well as the form of the parameter $\nu$ could be further relaxed. 
By Assumption~\ref{ass:1} we can find $\gamma \in \mathbb{R}$ such that
\begin{align} \label{eq:spectral}
\nu\|\Delta z\|^2_{L^2(\Omega)} + 2 \langle \mathcal{F}(z) ,z\rangle \ge (C^* + \gamma) \|z\|_{L^2(\Omega)}^2 
\qquad \forall z \in H_n^2(\Omega),
\end{align}
holds for the constant $C^* = \frac{3}{2}(R^2 + (3 R^2)^{4/3} \nu^{-1/3}+\nu^{-1}) + 1$ related to the bounds of the target trajectory $y_r$ and the model parameter $\nu$. The specific form of this constant will become clear from our analysis below.
By non-negativity of the left-hand side, any choice $\gamma \le - C^*$ is certainly admissible. 
This \emph{spectral estimate} is the key ingredient in stability analysis and allows us to establish our main result.
\begin{theorem} \label{thm:1}
Let Assumption~\ref{ass:1} hold with $\nu,R>0$. 
Then for any $T>0$ and any choice of $y_0\in L^2(\Omega)$, $h \in L^2(0,T;L^2)$, problem \eqref{eq:ch1}--\eqref{eq:ch2} has a weak solution 
\begin{align*}
y \in L^\infty(0,T;L^2) \cap L^2(0,T;H_n^2) 
\quad \text{with} \quad 
\partial_t y \in L^{4/3}(0,T;H_n^{-2})
\end{align*}
satisfying the initial condition $y(0)=y_0$ and the estimate 
\begin{align} \label{eq:est}
\|y_r(t)-y(t)\|_{L^2}^2 
\le e^{-\gamma t} \|y_r(0) - y(0) \|_{L^2}^2 + \int_0^t e^{-\gamma (t-s)} \|{{h(s) - h_r(s)}}\|_{L^2}^2 ds
\end{align}
for all $0 \le t \le T$ and any $\gamma$ satisfying condition~\eqref{eq:spectral}. 
Solutions thus exist globally in time and, if $\gamma > 0$ and $h=h_r$, they stabilise exponentially around the target trajectory $y_r$. 
\end{theorem}
\begin{remark} \label{rem:1}
Let us note that condition~\eqref{eq:spectral} is valid for $\mathcal{F} \equiv 0$ and $\gamma=-C^*$, which results in the global existence of a solution to the uncontrolled system. 
As we show in Section~\ref{sec:feedback}, condition \eqref{eq:spectral} can be satisfied for any desired $\gamma > 0$ if the feedback is chosen sufficiently strong in order to, roughly speaking, control the low-frequency components of the spectrum of the bi-Laplacian operator.  
We further emphasise that condition \eqref{eq:spectral} does not involve the solution $y$ of the system. The inequality \eqref{eq:est} thus yields a global stability estimate.
\end{remark}

\subsection{Outline of the manuscript}
A detailed proof of Theorem~\ref{thm:1} is given in Section~\ref{sec:proof}. 
It relies on careful energy estimates for the nonlinear system that describe the difference between controlled state $y$ and the target trajectory $y_r$. 
In Section~\ref{sec:feedback}, we construct explicit finite-dimensional feedback operators that satisfy \eqref{eq:spectral} with $\gamma>0$. This analysis is strongly inspired by \cite{AzouaniTiti2014,KunRodWal21}.
Due to the chosen functional analytic setting, we can allow for pointwise actuators and measurements. 
In Section~\ref{sec:disc}, we show that the stability results of Theorem~\ref{thm:1} carry over almost verbatim to appropriate discretisations of the problem. In particular, we obtain uniform stabilisation for discrete solutions independent of the discretisation parameters.
In Section~\ref{sec:num}, we present some numerical tests to illustrate the validity of our theoretical results and the importance of our assumptions.

\section{Proof of Theorem~~\ref{thm:1}}
\label{sec:proof}

The difference $z=y-y_r$ of solutions to \eqref{eq:ch1}--\eqref{eq:ch2} and \eqref{eq:t1}--\eqref{eq:t2} formally satisfies
\begin{alignat}{2}
\partial_t z + \nu \Delta^2 z - \Delta \varphi_r(z) &= g_r - \mathcal{F} z \qquad &&\text{in } \Omega, \ t>0, \label{eq:dif1}\\
\partial_n z = \partial_n \Delta z &= 0 \qquad && \text{in } \partial\Omega, \ t>0. \label{eq:dif2}
\end{alignat}
For ease of notation, we introduced the abbreviations $\varphi_r(z) = \varphi(y) - \varphi(y_r)$ and $g_r = h - h_r$. 
In the following, we prove the existence of a unique solution to this problem by Galerkin approximation~\cite{roubicek,wloka}.
Setting $y=y_r+z$, we then obtain a weak solution $y$ to \eqref{eq:ch1}--\eqref{eq:ch2}, 
and appropriate estimates for $z$ lead to \eqref{eq:est} and thus complete the proof of the theorem. 

\subsection{Weak formulation}
By testing \eqref{eq:dif1} appropriately, integration-by-parts, and the boundary conditions \eqref{eq:dif2}, any sufficiently smooth solution of \eqref{eq:dif1}--\eqref{eq:dif2} can be seen to solve 
\begin{alignat}{2} \label{eq:var}
\langle \partial_t z,v\rangle + \nu \langle \Delta z, \Delta v\rangle + \langle \mathcal{F} z, v\rangle &= \langle g_r ,v\rangle + \langle \varphi_r(z),\Delta v\rangle,
\end{alignat}
for all $v \in H_n^2$ and all $t \ge 0$ of relevance. 
Vice versa, any sufficiently smooth function $z$ satisfying \eqref{eq:var} for all $v \in H_n^2$  and $t \ge 0$ is a solution of \eqref{eq:dif1}--\eqref{eq:dif2}. This motivates
\begin{definition}
A function $z \in L^2(0,T;H_n^2) \cap L^\infty(0,T;L^2)$ with $\partial_t z \in L^{4/3}(0,T;H_n^{-2})$ that satisfies \eqref{eq:var} for all $v \in H_n^2$ and a.a. $0 \le t \le T$ is called a weak solution of \eqref{eq:dif1}--\eqref{eq:dif2}.
\end{definition}
%
The reduced integrability of the time derivative is due to the nonlinearity of the equation, which will become clear from the estimates given below.

\subsection{Galerkin approximation}

From compactness of the resolvent of $\Delta^2$, one can see that there exists a countable set $(\alpha_k,z_k)$, $k \in \mathbb{N}$, of solutions to 
\begin{align}
\langle \Delta z,\Delta v\rangle = \alpha \langle z,v\rangle  \qquad \forall v \in H_n^2.
\end{align}
The eigenvalues $\alpha_k \ge 0$ can be sorted in increasing order and the eigenfunctions can be chosen such that $\|\Delta z_k\|^2_{L^2} + \|z_k\|_{L^2}^2=1$ and $\langle \Delta z_k,\Delta z_m\rangle = \langle z_k,z_m\rangle=0$ for $m \ne k$. 
We define $V_N = \operatorname{span}\{z_k : k \le N\}$ and note that $V_N \subset V_M$ for $N \le M$ and $\overline{\bigcup_{N \in \mathbb{N}} V_N} = H_n^2$. 
The Galerkin approximation of \eqref{eq:var} reads
\begin{alignat}{2} \label{eq:varN}
\langle \partial_t z_N,v_N\rangle + \nu \langle \Delta z_N, \Delta v_N\rangle + \langle \mathcal{F} z_N, v_N\rangle &= \langle g_r ,v_N\rangle + \langle \varphi_r(z_N),\Delta v_N\rangle 
\end{alignat}
for all $v_N \in V_N$ and all $t \ge 0$ of interest. Note that the test function $v_N$ is independent of time. To fix the solution, we further require 
\begin{align} \label{eq:initN}
\langle z_N(0),v_N\rangle &= \langle z_0,v_N\rangle \qquad \forall v_N \in V_N,
\end{align}
with appropriate initial value $z_0 \in L^2(\Omega)$ given.
We then have the following result.
\begin{lemma}\label{lem:galerkin}
Let Assumption~\ref{ass:1} be valid. Then for any $z_{0} \in L^2(\Omega)$ and $g_r \in L^2(0,T;L^2)$ there exists a unique solution $z_N \in H^1(0,T;V_N)$ of \eqref{eq:varN}--\eqref{eq:initN} and 
\begin{align*}
\|z_N\|_{L^\infty(L^2)} + \|\Delta z_N\|_{L^2(L^2)} + \|\partial_t z_N\|_{L^{4/3}(H_n^{-2})} \le C_T (\|z_0\|_{L^2} + \|g_r\|_{L^2(L^2)}),
\end{align*}
with a constant $C_T$ that is independent of the approximation index $N$ and the data $z_0$, $g_r$. 
Further, with $\gamma$ from the spectral condition~\eqref{eq:spectral}, there holds
\begin{align} \label{eq:estN}
\|z_N(t)\|_{L^2}^2 &\le e^{-\gamma t} \|z_0\|^2_{L^2} + \int_0^t e^{-\gamma (t-s)} \|g_r\|_{L^2}^2 \, ds .
\end{align}
\end{lemma}
\begin{proof}
Since $V_N$ is finite-dimensional, the system \eqref{eq:varN}--\eqref{eq:initN} amounts to an initial value problem in $\mathbb{R}^N$ and the local existence of a unique solution follows readily from the Cauchy--Lipschitz theorem. 
We now establish the corresponding estimate. 
Choosing the test function as $v_N=z_N(t)$ in \eqref{eq:varN} yields
\begin{align}  \label{eq:enest}
\frac{d}{dt} \frac{1}{2}\| z_N\|^2_{L^2}
+ \nu \|\Delta z_N\|^2_{L^2} +  \langle \mathcal{F} z_N, z_N \rangle
= \langle g_r, z_N\rangle - \langle \nabla \varphi_r(z_N), \nabla z_N\rangle.
\end{align}
The integrand in the last term can be expressed as
\begin{align*}
-\nabla \varphi_r(z_N) \cdot \nabla z_N 
&= \varphi'(y_r) \nabla y_r \cdot \nabla z_N-\varphi'(y_r+z_N) \nabla (y_r+z_N) \cdot \nabla z_N  \\
&= (\varphi'(y_r)-\varphi'(y_r+z_{{N}}) ) \nabla y_r \cdot \nabla z_N - \varphi'(y_r+z_N) |\nabla z_N|^2.
\end{align*}
Using that $\varphi(y)={y^3-y}$, we further see that $\varphi'(y_r+z_N) = 3 (y_r+z_N)^2-1$ and hence
\begin{align*}
\varphi'(y_r)-\varphi'(y_r+z_N)  
=3 y_r^2-3 (y_r+z_N)^2=-3z_N(z_N+2y_r)= -3 {z_N} (z_N+y_r) - 3z_N  y_r. 
\end{align*}
By using Young's inequality, we may further estimate 
\begin{align} 
&-\nabla \varphi_r({z_N}) \cdot \nabla z_N  \notag \\
&=-3z_N(z_N+y_r)  \nabla y_r \cdot \nabla z_N-3z_Ny_r  \nabla y_r \cdot \nabla z_N - (3(y_r+z_N)^2-1)|\nabla z_N|^2 \notag
\\
&\leq 3(z_N+y_r)^2  |\nabla z_N|^2+\frac{3}{4} |z_N|^2  |\nabla y_r|^2 -3z_Ny_r  \nabla y_r \cdot \nabla z_N - (3(y_r+z_N)^2-1)|\nabla z_N|^2 \notag
\\
&\leq \frac{3}{4} |\nabla y_r|^2 |z_N|^2 + 3|y_r|\, |\nabla y_r|\, |z_N|\, |\nabla z_N| + {|\nabla z_N|^2}. \label{Eq:EstimatePhi}
\end{align}
We then integrate this expression over $\Omega$, use Hölder inequalities, the bounds for $y_r$, see Assumption \ref{ass:1}, and the fact that $\|\nabla z\|_{L^2}^2 \le \|z\|_{L^2} \|\Delta z\|_{L^2}$ for all $z \in H_n^2$. This leads to
\begin{align*}
-\langle \nabla \varphi_r(z_N),  \nabla z_N \rangle
&\le \frac{3}{4} R^2 \|z_N\|_{L^2}^2+ 3 R^2 \|z_N\|_{L^2}^{3/2} \|\Delta z_N\|_{L^2}^{1/2}+{\|z_N\|_{L^2}\|\Delta z_N\|_{L^2}} \\
&\le \frac{3}{4}( R^2 + (3 R^2)^{4/3} \nu^{-1/3}+{\nu^{-1}}) \|z_N\|_{L^2}^2 + \frac{\nu}{{2}} \|\Delta z_N\|^2_{L^2}.
\end{align*}
Plugging this into \eqref{eq:enest} and estimating the remaining term $\langle g_r,z_N\rangle$ in \eqref{eq:enest}  yields  
\begin{align}\label{eq:spectral-Gal}
\frac{d}{dt} \frac{1}{2}\| z_N\|^2_{L^2}
+ \frac{\nu}{2} \|\Delta z_N\|^2_{L^2} +  \langle \mathcal{F} z_N, z_N \rangle
\le \frac{1}{2} C^* \|z_N\|_{L^2}^2 + \frac{1}{2} \|g_r\|_{L^2}^2,
\end{align}
with the constant $C^*$ as announced below condition~\eqref{eq:spectral}.
Using this estimate and applying Grönwall's inequality, we immediately obtain \eqref{eq:estN}.
By dropping the term $\langle \mathcal{F} z_N,z_N \rangle \ge 0$ and again applying Grönwall's inequality, we obtain the uniform bounds for $\|z_N\|_{L^\infty(L^2)}$ and $\|\Delta z_N\|_{L^2(L^2)}$. 
By definition of the dual norm, we have 
\begin{align} \label{eq:dual}
\|\partial_t z_N\|_{H_n^{-2}} 
&= \sup_{v \in H_n^2} \frac{|\langle \partial_t z_N,v\rangle|}{\|v\|_{H^2}} = \sup_{v_N \in V_N} \frac{|\langle \partial_t z_N,v_N\rangle|}{\|v_N\|_{H^2}}. 
\end{align}
In the last step, we used the orthogonality of the eigenfunctions spanning $V_N$. 
For the numerator in the last expression, we can use \eqref{eq:varN}, which yields 
\begin{equation} \label{Eq:ImprovableTimeReg}\begin{aligned}
\langle \partial_t z_N, v_N \rangle 
&= -\nu \langle \Delta z_N, v_N \rangle - \langle \mathcal{F}(z_N),v_N\rangle + \langle  \varphi_r(z_N),\Delta v_N\rangle + \langle  g_r,z_N \rangle \\
&\le C (1+ \|z_N\|_{H^2} + \|z_N\|_{L^2}^{3/2} \|{z_N}\|_{H^2}^{3/2} + \|g_r\|_{L^2}) \|v_N\|_{H^2},
\end{aligned}\end{equation}
where we used the Hölder and Gagliardo-Nirenberg inequalities, see \cite[Theorem 1.24]{roubicek}, to bound $\|\varphi_r(z_N)\|_{L^2} \le C' (1+\|z_N\|_{L^6}^3)$ and $\|z_N\|_{L^6} \le C''\|z_N\|_{L^2}^{1/2} \|z_N\|_{H^2}^{1/2}$.
Note that the constants $C$, $C'$, $C''$ in these estimates are independent of $N$.  Inserting in \eqref{eq:dual} yields the remaining bound for the time derivative. 
\end{proof}

\subsection{Passage to the limit}
The Galerkin approximations $z_N$ obtained in Lemma~\ref{lem:galerkin} are uniformly bounded. We may thus extract a subsequence (again denoted by $z_N$) with limit $z \in L^\infty(0,T;L^2) \cap L^2(0,T;H_n^2) \cap W^{1,4/3}(0,T;H_n^{-2})$ such that 
\begin{alignat*}{5}
z_N &\rightharpoonup z \qquad && \text{in } L^2(0,T;H_n^2), \\
\partial_t z_N &\rightharpoonup \partial_t z \qquad && \text{in } L^{4/3}(0,T;H_n^{-2}), \\
z_N &\rightharpoonup^* z \qquad && \text{in } L^\infty(0,T;L^2),
\intertext{with $N \to \infty$. 
By the Aubin-Lions compactness lemma, we further observe that}
z_N &\to z \qquad && \text{strongly in } L^4(0,T;L^3).
\end{alignat*}
To see this, we apply \cite[Lemma~7.8]{roubicek} with $V_1=H_n^2$, $V_2=H^1$, $V_3=H_n^{-2}$, $p=2$, $q=4/3$, $V_4=L^3$, $H=L^2$, $\lambda=1/2$, which shows that the space
$$ 
(W^{1,2,4/3}(0,T;H_n^2,H_n^{-2}) = \{v \in L^2(0,T;H_n^2) : \partial_t v \in L^{4/3}(0,T;H_n^{-2})\}) \cap L^\infty(0,T;L^2)
$$
compactly embeds into $L^{4}(0,T;L^3)$.
Since $\varphi_r$ involves at most cubic nonlinearities, we can deduce that $\varphi_r(z_N) \to \varphi_r(z)$ in $L^1(0,T;L^1)$. 
We can then pass to the limit in all terms of the weak formulation \eqref{eq:varN} and see that the limit $z$ satisfies \eqref{eq:var}. 
By the continuity of the temporal traces in~$H_n^{-2}$, we further obtain $z(0)=z_0$. 
We thus obtain the following result.
\begin{lemma}
Let the assumptions of Theorem~\ref{thm:1} be valid. 
Then for any $z_0 \in L^2$ and $g_r \in L^2(0,T;L^2)$, 
problem
\eqref{eq:dif1}--\eqref{eq:dif2} has a global weak solution with  $z(0)=z_0$ and
\begin{align} \label{eq:est_dif}
\|z(t)\|_{L^2}^2 
&\le e^{-\gamma t} \|z_0\|^2_{L^2} + \int_0^t e^{-\gamma (t-s)} \|{{g_r}}\|_{L^2}^2 \, ds .
\end{align}
\end{lemma}
\begin{proof}
The existence of a weak solution already follows from the arguments outlined above. Furthermore, the estimate \eqref{eq:estN} transfers to the limit $z = \lim_N z_N$ by continuity.    
\end{proof}

\subsection{Proof of Theorem~\ref{thm:1}.}
Let $z_0=y_0-y_r(0)$ and $z$ be a weak solution of \eqref{eq:dif1}--\eqref{eq:dif2}
with $z(0)=z_0$.
Then by construction, the function $y=y_r+z$ is a weak solution of problem~\eqref{eq:ch1}--\eqref{eq:ch2} with $y(0)=y_0$.
The estimate \eqref{eq:est} for this solution follows from \eqref{eq:est_dif}. 
\qed

\subsection{Uniqueness} \label{subsec:Unique}
For completeness of the presentation, let us make some remarks concerning the uniqueness of the solution. We begin with a weak-strong uniqueness result. 
\begin{theorem} \label{thm:unique}
Let $y_1,y_2$ denote two weak solutions of the problem \eqref{eq:ch1}--\eqref{eq:ch2} with the same initial values. 
In addition, assume that $y_1,y_2 \in H^1(0,T;H_n^{-2})$ and 
that $y_2 \in L^4(0,T;W^{1,d})$ where $1 \le d \le 3$ is the spatial dimension.
Then $y_1 = y_2$. 
\end{theorem}
\begin{proof}
We start with the case $d=3$. 
The function $z=y_1-y_2$ solves
\begin{align*}
\langle \partial_t z,v\rangle + \nu \langle \Delta z, \Delta v\rangle + \langle \mathcal{F} z, v\rangle &= \langle \varphi_2(z),\Delta v\rangle  
\end{align*}
for all $v \in H_n^2(\Omega)$ and a.a.~$t \ge 0$ with $\varphi_2(z) = \varphi(y_1) - \varphi(y_2)$. With similar arguments as used in the proof of Lemma~\ref{lem:galerkin}, one can see that 
\begin{align}
\frac{d}{dt} \frac{1}{2} \|z\|^2_{L^2} &+ \nu \|\Delta z\|_{L^2}^2 + \langle \mathcal{F} z,z\rangle  
\le C (\|y_1\|_{L^6} +\|y_2\|_{L^6}) \|\nabla y_2\|_{L^3} \|z\|_{L^6} \|\nabla z\|_{L^2} + \|\nabla z\|^2_{L^2} \notag \\
&\le C' (1 + \|y_1\|_{H^1}^2 + \|y_2\|_{H^1}^2) \|\nabla y_2\|^2_{L^3} \|z\|_{L^2}^2 + \nu \|\Delta z\|_{L^2}^2, \label{eq:unique1}
\end{align}
where we tacitly employed the additional regularity of the time derivatives. 
For the second step, we further used that $\|y\|_{L^6} \le C \|y\|_{H^1}$, $\|\nabla z\|_{L^2} \le \|z\|_{L^2}^{1/2} \|\Delta z\|_{L^2}^{1/2}$, and Young inequalities. From the regularity of weak solutions, we can see that $y_1$, $y_2 \in L^4(H^1)$ and further $\nabla y_2 \in L^4(L^3)$. This shows that  $$\alpha(t) :=(1 + \|y_1\|_{H^1}^2 + \|y_2\|_{H^1}^2) \|\nabla y_2\|^2_{L^3}$$ is integrable in time. We may then absorb the term $\nu \|\Delta z\|^2_{L^2}$ in the left-hand side of the previous estimate \eqref{eq:unique1} and apply Grönwall's inequality to see that $$\|z(t)\|_{L^2}^2 \le e^{\int_0^t \alpha(s) ds} \|z(0)\|_{L^2}^2 = 0,$$ by assumption on the initial values.
In dimension $d \le 2$, we can use slightly different estimates to obtain the same bound with $\alpha(t) = (1 + \|y_1\|_{H^1}^2 + \|y_2\|_{H^1}^2) \|\nabla y_2\|^2_{L^2}$ instead. 
\end{proof}

\begin{remark}
In dimension $d \le 2$, any pair of weak solutions to \eqref{eq:ch1}--\eqref{eq:ch2} satisfies the conditions of Theorem~\ref{thm:unique}. 
To see this, we invoke the Gagliardo-Nirenberg inequality, which yields $\|y\|_{L^6} \leq  \|y\|^{2/3}_{L^2} \|y\|^{1/3}_{H^2}$, and hence 
$$
\langle \varphi(y),\Delta v\rangle \le C (1+\|y\|_{L^6}^3) \|\Delta v\|_{L^2} 
\le C' (1+ \|y\|_{L^2}^4) \|y\|_{H^2}^2.
$$ 
This suffices to show that $\partial_t y \in L^2(0,T;H_n^{-2})$. Further, the condition $y \in L^4(0,T;H^1)$ is valid for any weak solution $y \in L^2(0,T;H_n^2) \cap L^\infty(0,T;L^2)$. 
In dimension $d=3$, the regularity conditions of Theorem~\ref{thm:unique} can be established if $\mathcal{F}$ is bounded as a mapping from $L^2$ to $L^2$, and the domain $\Omega$ has a sufficiently regular boundary. In this case, the required regularity of the solution can be obtained by testing \eqref{eq:ch1}--\eqref{eq:ch2} with $v=\Delta y$. 
We refer to \cite{Boyer1999} 
for some typical results in this direction. 
\end{remark}

\section{Construction of feedback operators}
\label{sec:feedback}

A possible choice of a feedback operator that satisfies \eqref{eq:spectral} for given $\gamma>0$ would be $\mathcal{F} u = \lambda u$ with $\lambda$ sufficiently large. However, such a construction is impracticable from an application point of view. 
Using similar arguments as in \cite{AzouaniTiti2014}, we now show that finitely many measurements and actuators are sufficient to satisfy \eqref{eq:spectral}. 
Alternative constructions can be found in \cite{KunRodWal21,KunRodWal24-cocv}. 
In the following, we consider feedback operators of the form
\begin{align} \label{eq:feedback}
\mathcal{F} z\coloneqq\lambda \sum\nolimits_T |T| \, \langle \Phi_T,z\rangle\Phi_T
\end{align}
where $\mathcal{T}_H = \{T\}$ is a partition of the domain $\Omega$ into quasi-uniform subdomains of size~$H$. Here, $|T|=\int_T 1 \, dx$ denotes the volume of $T$ and $\Phi_T$ represent scaled characteristic functions or delta distributions supported in $T$. 
By quasi-uniformity, we mean that each $T$ contains a ball of size $c H$ and is contained in a ball of size $H$ with uniform constant $c>0$.
Note that $\mathcal{F} z$ only depends on the finite-dimensional measurements $\langle \Phi_T,z\rangle$, so $\mathcal{F}$ defines a static output feedback control operator.

\subsection{Pointwise control}
We choose $\Phi_T =\delta_{x_T}$ as delta distribution supported in $x_T \in T$. Then the action of the feedback operator can be written as 
\begin{align} \label{eq:pointwise}
\langle \mathcal{F} z, v \rangle 
= \lambda \sum\nolimits_T |T| \, z(x_T) \, v(x_T),
\end{align}
which can be interpreted as a numerical approximation of the integral $\langle z,v\rangle_{L^2}$. 
Let us denote by $(I_H^0 z)(x) = \sum_T z(x_T) \chi_T(x)$ the piecewise constant interpolant, which is well-defined for functions $z \in H_n^2$. 
We may then rewrite \eqref{eq:pointwise} as
\begin{align*}
\langle \mathcal{F} z, v\rangle 
=\lambda \langle I_H^0 z, I_H^0 v\rangle_{L^2(\Omega)}.
\end{align*}
From standard approximation error estimates \cite{BrennerScott2008}, 
one can deduce that 
\begin{align*}
\|z-I_H^0 z\|_{L^2} 
\le C' H (\|z\|_{L^2} + \|\Delta z \|_{L^2}) \qquad \forall z \in H_n^2(\Omega). 
\end{align*}
%
Using triangle inequalities and elementary manipulations, we can then further estimate 
\begin{align*}
\frac{1}{2}\|z\|_{L^2}^2 
& \le \|I_H^0 z\|^2_{L^2} + \|z - I_H^0 z\|_{L^2}^2 
\le \|I_H^0 z\|^2_{L^2} + C'' H^2 (\|z\|^2_{L^2} + \|\Delta z\|^2_{L^2}). 
\end{align*}
By assuming $C'' H^2 \le 1/4$, which is not really restrictive, we may absorb the second term on the left-hand side. After multiplying with $2 \lambda$, we thus obtain 
\begin{align*}
\frac{\lambda}{2} \|z\|_{L^2}^2 
&\le 2 \langle \mathcal{F} z, z\rangle + 2 C'' \lambda H^2 \|\Delta z\|^2_{L^2} \qquad \forall z \in H_n^2(\Omega). 
\end{align*}
We can then choose $\lambda$ sufficiently large and $H$ sufficiently small, such that $\lambda \ge 2 (C^* + \gamma)$
and additionally $2 C'' \lambda H^2 \le \frac{3 \nu}{2}$,
which yields the spectral inequality \eqref{eq:spectral} as desired. 

\subsection{Further constructions and monotonicity}
Instead of \eqref{eq:pointwise}, we may also consider a weighted $\langle \mathcal{F} z, v \rangle = \sum_T |T| z(x_T) v(x_T) \beta_T$ with $0 < \underline \beta \le w_T \le \overline \beta$ uniformly bounded. One may also consider nonlocal interaction $\langle \mathcal{F} z,v\rangle = \sum_T \sum_{T'} |T| |T'| z(x_T) v(x_{T'}) \beta_{T,T'}$, if the matrix generated by the weights $\beta_{T,T'}$ is symmetric with uniformly positive and bounded eigenvalues. 
Similar estimates can be applied if we choose the actuators as indicator functions~$\Phi_T=\frac{1}{|\omega_T|}\chi_{\omega_T}$, where $\chi_{\omega_T}$ are the indicator functions for appropriate domains $\omega_T \subset T$. In this case, the action of the feedback operator reads
\begin{align*}
\langle \mathcal{F} z,v\rangle = \lambda \sum\nolimits_T |T| \, \bar z(x_T) \, \bar v(x_T),
\end{align*}
where $\bar z(x_T)=\frac{1}{|\omega_T|} \int_{\omega_T} z(x)\,{\rm d} x$ is a local average of the function $z$.
We may then introduce the quasi-interpolation operator $\bar I_H^0 z = \sum_T \frac{1}{|\omega_T|} \, \bar z(x_T) \, \chi_{\omega_T}$, which has approximation properties similar to the nodal interpolation $I_h^0$ studied above. 
The spectral inequality \eqref{eq:spectral} can thus again be established by similar reasoning as before. 
Note that here $\mathcal{F} : L^2 \to L^2$ is continuous, which allows us to sharpen the previous estimates to some extent. 
For further constructions, which can be analyzed in a similar manner, we refer to \cite{KunRodWal21,KunRodWal24-cocv}.

\medskip 

Before concluding this section, let us  highlight some basic monotonicity properties of the feedback operators constructed above, which follow by elementary arguments. 
\begin{remark}[Monotonicity] \label{rem:monotonicity}
Let $\mathcal{F}_* : H_n^2 \to H_n^{-2}$ be a positive semi-definite feedback operator, as required in our analysis and the examples above, and let $\lambda_*$ be sufficiently large such that \eqref{eq:spectral} holds. 
Then \eqref{eq:spectral} remains valid for all $\lambda \ge \lambda_*$ and all feedback operators $\mathcal{F} \ge \mathcal{F}_*$., i.e.,  which satisfy
$\langle \mathcal{F} z,z\rangle \ge \langle \mathcal{F}_* z,z\rangle$ for all $z \in H_n^2$.
\end{remark}
\noindent
These observations are useful for interpreting the numerical results later on.

\section{Discretisation}
\label{sec:disc}

The proof of Theorem~\ref{thm:1} presented above mainly relies on energy estimates and the spectral inequality \eqref{eq:spectral}, which both carry over verbatim to Galerkin approximations in space; see Section~\ref{sec:proof}. Together with an appropriate time discretisation strategy, one can thus obtain fully discrete schemes that inherit exponential stabilizability towards target trajectories.   
We state and prove our main result for a general class of discretisation schemes and later present one particular example used in our numerical tests.

\subsection{Discretisation scheme}
Let $V_h \subset H_n^2(\Omega)$ be finite-dimensional and $\tau>0$ a given time step size. We set $t^n = n \tau$ for $n \ge 0$ and abbreviate by $d_\tau a = \frac{1}{\tau} (a^n - a^{n-1})$ the backward difference quotient. 
The discrete solution $(y_h^n)_{n \ge 0} \subset V_h$ is defined by 
\begin{align} \label{eq:ch1h}
\langle d_\tau y_h^n, v_h\rangle + \langle \nu \Delta y_h^n, \Delta v_h\rangle &= \langle \varphi(y_h^n), \Delta v_h\rangle +\langle h^n,v_h \rangle - \langle \mathcal{F}(y_h^n- y_{r,h}^n),v_h\rangle
\end{align}
for all $v_h \in V_h$ and $t^n>0$ with 
initial condition
\begin{align} \label{eq:ch2h}
\langle y_h^0,v_h\rangle &= \langle y_0,v_h\rangle \qquad \forall v_h \in V_h.
\end{align}
Furthermore, $h^n = h(t^n)$ is the evaluation at time $t^n$ and $\{y_{r,h}^n\} \subset V_h$ is a given discrete reference trajectory, e.g. obtained by interpolation of $y_r$.  
This scheme amounts to a discretisation of \eqref{eq:ch1}--\eqref{eq:ch2} by Galerkin approximation in space and the implicit Euler method in time. The feasibility of the method follows by standard arguments. 
\begin{lemma}
Let Assumption~\ref{ass:1} hold, $V_h \subset H_n^2$ be a finite-dimensional sub-space, and $\tau >0$. Then the scheme \eqref{eq:ch1h}--\eqref{eq:ch2h} admits at least one solution.    
\end{lemma}
\begin{proof}
The initial condition \eqref{eq:ch2h} amounts to an elliptic variational problem which, according to the Lax-Milgram lemma, has a unique solution $y_h^0 \in V_h$. 
Now, let $y_h^{n-1}$, $n \ge 1$ be given. Then \eqref{eq:ch1h} can be restated as a nonlinear system $G^n(\bar y^n) = 0$ in $\mathbb{R}^N$ with $\bar y^n \in \mathbb{R}^N$ denoting the coordinate vector representing $y_h \in V_h$ in some basis of $V_h$. 
From the properties of the nonlinear function $\varphi(y)$, one can deduce that 
\begin{align*}
\langle G^n(\bar y^n), \bar y^n \rangle \ge 0 \qquad \text{for all } |\bar y^n| \ge R
\end{align*}
with $R$ sufficiently large. From the lemma on the zeros of vector fields \cite[Chapter 9.1]{Evans}, we thus deduce that there exists $\bar y^n \in \mathbb{R}^N$ with $|\bar y^n| \le R$ such that $G^n(\bar y^n) = 0$. 
\end{proof}

\begin{remark}
The uniqueness of the discrete solution can be shown under certain restrictions on the time step $\tau \leq \tau_0(h)$. Since all further results are valid for every discrete solution, we do not go into details here, but refer to \cite{brunk2023stability} for results in this direction. 
\end{remark}

\subsection{Discrete exponential stability}
We now study the stabilisation of discrete solutions. 
To do so, we consider the discrete target trajectory as a solution of 
\begin{align} \label{eq:ch1hr}
\langle d_\tau y_{r,h}^n, v_h\rangle + \langle \nu \Delta y_{r,h}^n, \Delta v_h\rangle -  \langle \varphi(y_{r,h}^n), \Delta v_h\rangle &= \langle h_r^n,v_h \rangle
\end{align}
with appropriate right hand side $h_r^n \in L^2(\Omega)$. 
Typical choices for $y_{r,h}^n$ we have in mind are steady states of the uncontrolled discrete system or discrete trajectories emerging from different initial values. In this case $h_r^n=h^n$ for all $n \ge 1$. 
With almost verbatim arguments as used for the proofs of Section~\ref{sec:proof}, we can then derive the following result.
\begin{theorem}
Let $\{y_{r,h}^n\} \subset V_h$ be a given target trajectory satisfying \eqref{eq:ch1hr} for all $n \ge 1$ and $\|y_{r,h}^n\|_{W^{1,\infty}} \le R$.
Further, let Assumption~\ref{ass:1} and let condition \eqref{eq:spectral} hold with some $\gamma > 0$. 
Then any solution $\{y_h^n\} \subset V_h$ of \eqref{eq:ch1h} satisfies
\begin{align*}
\|y_h^n - y_{r,h}^n\|_{L^2}^2 \leq e^{-\tilde \gamma t^n} \|y_h^0 - y_{r,h}^0\|_{L^2}^2 + \tau \sum\nolimits_{k=0}^n \tau e^{-\tilde \gamma (t^n - t^k)} \|h^k - h_r^k\|_{L^2}^2,
\end{align*}
for all $n \geq 0$ with $\tilde \gamma$ defined by 
$\tau \tilde \gamma = \log(1+\tau \gamma)$. 
The solution of the controlled discrete system thus stabilises exponentially around the discrete target trajectory. 
\end{theorem}
\begin{proof}
Let $z_h^n := y_h^n - y_{r,h}^n$ denote the difference to the target values. Then 
\begin{align*} 
\langle d_\tau z_h^n, v_h\rangle + \langle \nu \Delta z_h^n, \Delta v_h\rangle &= \langle \varphi_{r,h}(z_h^n), \Delta v_h\rangle +\langle g_{r,h}^n,v_h \rangle
\end{align*}
with $\varphi_{r,h}(z) = \varphi(y_h) - \varphi(y_{r,h})$ and $g_{r,h}^n = h^n - h_r^n$. By testing this discrete variational identity with $v_h=z_h^n$ and using that $d_\tau \frac{1}{2} \|z_h^n\|_{L^2}^2 \le \langle d_\tau z_h^n,z_h^n \rangle$, we can see that 
\begin{align} \label{eq:est1}
d_\tau \frac{1}{2} \|z_h^n\|^2_{L^2} + \nu \|\Delta z_h^n\|^2_{L^2} 
+ \langle \mathcal{F} z_h^n, z_h^n\rangle 
\le \langle g_{r,h}^n,z_h^n\rangle - \langle \nabla \varphi_{r,h}(z_h^n),\nabla z_h^n\rangle. 
\end{align}
With the very same reasoning as in employed in Section~\ref{sec:proof}, we obtain the estimate
\begin{align*} 
-\langle \nabla \varphi_{r,h}({z_h^n}),  \nabla z_h^n \rangle
&\le \frac{3}{4}( R^2 + (3 R^2)^{4/3} \nu^{-1/3}+{\nu^{-1}}) \|z_h^n\|_{L^2}^2 + \frac{\nu}{{2}} \|\Delta z_h^n\|^2_{L^2}.
\end{align*}
The last term can be absorbed into the left hand side of 
\eqref{eq:est1}, and the remaining term in this estimate can be bounded by Young's inequality. In summary, this leads to
\begin{align*}
d_\tau \|z_h^n\|^2_{L^2} + \nu \|\Delta z_h^n\|_{L^2}^2  + 2 \langle \mathcal{F} z_h^n,z_h^n\rangle 
\le \|g_{r,h}^n\|_{L^2}^2 + C^* \|z_h^n\|_{L^2}^2,
\end{align*}
where $C^*= \frac{3}{2}( R^2 + (3 R(R+1))^{4/3} \nu^{-1/3}+\nu^{-1}) + 1$. 
Using the spectral bound \eqref{eq:spectral}, the definition of the backward difference quotient $d_\tau$, and multiplying by $\tau$ then leads to 
\begin{align*}
\|z_h^n\|^2_{L^2} + \gamma \tau  \|z_h^n\|^2_{L^2} \le \|z_h^{n-1}\|^2_{L^2} + \tau \|g_{r,h}^n\|_{L^2}^2.  
\end{align*}
The claim of the theorem now follows by recursive application of this inequality and using that $\frac{1}{1+\tau \gamma} = e^{-\tilde \gamma \tau}$, which follows from the definition of $\tilde \gamma$.
\end{proof}

\begin{remark}\label{rem:gamma-tilgamma}
We emphasise that the discrete stability estimate is derived under essentially the same conditions as used on the continuous level. In particular, the same spectral estimate is employed in the stability proof. 
In addition, note that for $\tilde \gamma \to \gamma$ for $\tau \to 0$. Hence $\tilde \gamma$ is essentially independent of the time-step size.
Moreover, the estimates of the previous theorem are also independent of spatial discretisation. 
They would also hold for a semidiscrete approximation in time, in which case we may choose $y_{r,h} = y_r(t^n)$. Then $h_r^n = d_\tau y_r^n - \partial_t y_r(t^n)$ which converges to zero with $\tau \to 0$ if $y_r$ is sufficiently smooth in time. 
In summary, we thus expect to observe discretisation-independent stabilisation towards sufficiently regular discrete target trajectories.
\end{remark}

\section{Numerical validation}
\label{sec:num}
%
To illustrate our theoretical findings, we now present some numerical tests. For ease of presentation, we consider a two-dimensional domain $\Omega = (0,1)^2$. As discretisation of our model problem \eqref{eq:ch1}--\eqref{eq:ch2}, we consider the method \eqref{eq:ch1h} with the approximation space
\begin{align}
V_h = \{v_h \in H^2_n(\Omega) : v_h |_T \in P_5(T) \ \forall T \in \mathcal{T}_h,  \quad \partial_n v_h|_e \in P_3(e) \ \forall e \in \mathcal{E}_h \}.
\end{align}
Here, $\mathcal{T}_h = \{T\}$ is a quasi-uniform triangulation of the domain $\Omega$ and $\mathcal{E}_h = \{e\}$ denotes the set of edges. This is the finite element space made up of the \emph{Bell triangle}; see \cite{bell1969refined,Ciarlet2002} for details.
An implementation of the element is available in the finite element library \emph{Firedrake} \cite{rathgeber2016firedrake}. 
We consider pointwise feedback operators of the form 
\begin{align} \label{eq:feedback2}
\mathcal{F}_{\lambda,M} z = \frac{\lambda}{M^2} \sum_{j,k=1}^M z(\xi_j,\xi_k) \delta_{(\xi_j,\xi_k)}
\end{align}
with equidistant interpolation points $ \xi_i = (i-1/2) H$ and $H=1/M$; see Figure~\ref{fig:feedback}. This is exactly the setting discussed in Section~\ref{sec:feedback}. From our investigations, we know that for $\lambda$ and $M$ sufficiently large, the spectral condition \eqref{eq:spectral} can be satisfied for any $\gamma>0$. 
\begin{figure}[ht!]
	\centering
\includegraphics[width=.9\textwidth,page=1]{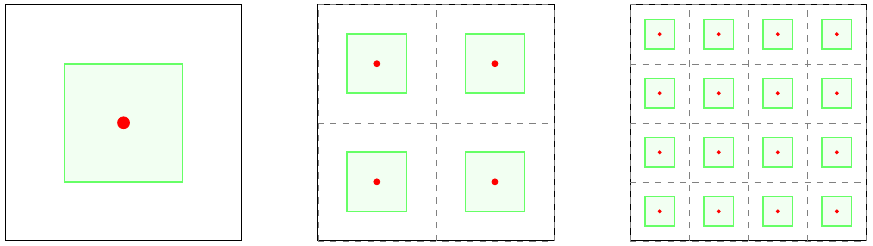}
\caption{Arrangement of feedback measurements and actuators for a domain $\Omega=(0,1)^2$ with $M=1,2,4$ subintervals per coordinate direction. Red dots correspond to pointwise actuators and measurements whereas green areas depict corresponding characteristic functions. \label{fig:feedback}}
\end{figure}

\subsection{Validation of spectral estimates}
\label{sec:num_spectral}
In a first series of tests, we want to verify the spectral estimate \eqref{eq:spectral} for the feedback operator \eqref{eq:feedback2} with parameters $\lambda$ and $M$ sufficiently large. 
Before going to the numerical tests, let us briefly comment on the expectations that one can have. Consider the self-adjoint elliptic eigenvalue problem
\begin{alignat*}{2}
\nu \Delta^2 v &= \alpha v \qquad && \text{in } \Omega \\
\partial_n v &= \partial_n \Delta v = 0 \qquad && \text{on } \partial\Omega.
\end{alignat*}
From the compactness of the resolvent of $\Delta^2$, we conclude that there exists a monotonically increasing sequence of eigenvalues $\alpha_n \ge 0$ that accumulates at infinity. By Fourier analysis, one can further see that $\alpha_n \approx \nu n^2$, since $\Omega$ is a two-dimensional domain. 
In order for \eqref{eq:spectral} to be valid with $\gamma \ge 0$, the feedback operator $\mathcal{F}$ must control at least the eigenmodes for $\nu \alpha_n \lesssim C^*$. To be able to satisfy this condition, one has to choose
\begin{align} \label{eq:estimate}
M^2 = n \gtrsim \sqrt{C^*/\nu}. 
\end{align}
By increasing $\lambda$ in \eqref{eq:feedback2}, we can further improve the strength of the feedback. 
For a given parameter $C^*$, depending on $R$ and $\nu$, we thus expect a certain trade-off between $M$ and $\lambda$ that allows us to guarantee the spectral estimate \eqref{eq:spectral} with $\gamma \ge 0$. 

\subsection*{Results.}
For given $(M,\lambda)$ we consider the generalized eigenvalue problem
\begin{align*}
  \nu (\Delta u,\Delta v)_{L^2}+2 \langle\mathcal{F}_{\lambda,M}u,v\rangle=\alpha(u,v)_{L^2},
\end{align*}
and compute the smallest eigenvalue $\alpha_{min}(M,\lambda)$ using SLEPc’s Krylov-Schur method.
In Figure~\ref{fig:spectral}, we illustrate the corresponding results by plotting the isolines of $\alpha_{\min}(M,\lambda)=C$ as a function of $M$ and $\lambda$ for different values of $C$.  
\begin{figure}[ht!]
	\centering
\includegraphics[width=.99\textwidth,page=2]{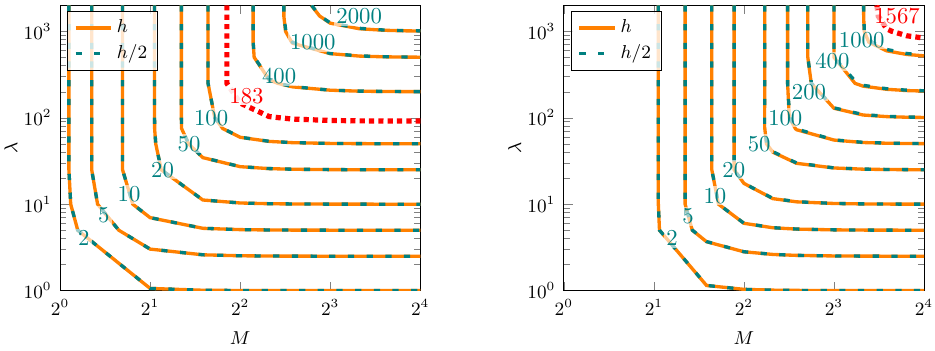}
\caption{Contour plots of the minimal eigenvalue $\alpha_{\min}(M, \lambda)$ for two spatial mesh sizes $h$ (solid lines) and $h/2$ (dashed lines), with parameter $\nu \in\{0.01,0.001\}$; the red dotted line indicates the theoretical threshold value $C^*$ for stabilisation where we chose $R=1$. \label{fig:spectral}} 
\end{figure}
As described in Remark~\ref{rem:monotonicity}, we observe monotonic dependence of $\alpha_{min}(M,\lambda)$ on $M$ and $\lambda$ when the value of $\nu>0$ is fixed. 

For $R=1$ and $\nu=0.01$, we obtain $C^*\approx 183$ in \eqref{eq:spectral}, while for $R=1$ and $\nu=0.001$, we get $C^*\approx 1567$. 
In Figure~\ref{fig:spectral}, we show the corresponding levels of $\alpha_{min}(M,\lambda)$ in red. As predicted by our preliminary considerations \eqref{eq:estimate} above, we see a moderate increase in the smallest number $M$ of actuators that allows to guarantee stability when decreasing $\nu$. At the same time, we observe $\lambda \approx 1/\nu$, which again seems clear from \eqref{eq:spectral}.  
The results depicted in Figure~\ref{fig:spectral} further demonstrate mesh-independence, which can be expected from the strong convergence estimates for computation of elliptic eigenvalue problems, and the high-order approximation provided by the Bell triangle~\cite{Babuska1989,Ciarlet2002}.

\subsection{Stabilisation of the Cahn-Hilliard equation}
We now turn to our model problem \eqref{eq:ch1}--\eqref{eq:ch2} with feedback stabilisation.  
The two functions
\begin{align} \label{eq:num}
y_r(x,y) = 0
\qquad \text{and} \qquad 
y_0(x,y) = \tanh\Big(\frac{2x-1}{\sqrt{8\nu}}\Big).
\end{align}
correspond to an unstable equilibrium and an approximation of the stable equilibrium of the uncontrolled system \eqref{eq:ch1}--\eqref{eq:ch2} with $\mathcal{F}=0$ and $h=0$. 
To see this, simulations can be run for the uncontrolled system with initial values $y(0)=y_0 + \delta y$ and $y(0)=y_r + \delta y$, with $\delta y$ denoting some small random perturbation. 
In our numerical tests, we start from $y(0)=y_0$ and try to steer the system through feedback into the unstable equilibrium $y_r$. 
Here $h=h_r=0$, and from Theorem~\ref{thm:1}, we expect exponential convergence if the feedback $\mathcal{F}$ is sufficiently strong, such that \eqref{eq:spectral} holds with $\gamma>0$. This can be achieved by operators $\mathcal{F}_{\lambda,M}$ as investigated above. 
For our simulations, we set $\nu=0.01$ and choose $h=2^{-5}$, $\tau=10^{-3}$ for the discretisation parameters. 

\subsection*{Results.}
In Figure~\ref{Fig:Simulation}, we show some snapshots of a solution to \eqref{eq:ch1}--\eqref{eq:ch2} obtained for simulations of the controlled system with $\mathcal{F}=\mathcal{F}_{\lambda,M}$ for different choices of $M$ and $\lambda$. 
\begin{figure}[ht!]
	\centering
\includegraphics[width=.99\textwidth,page=3]{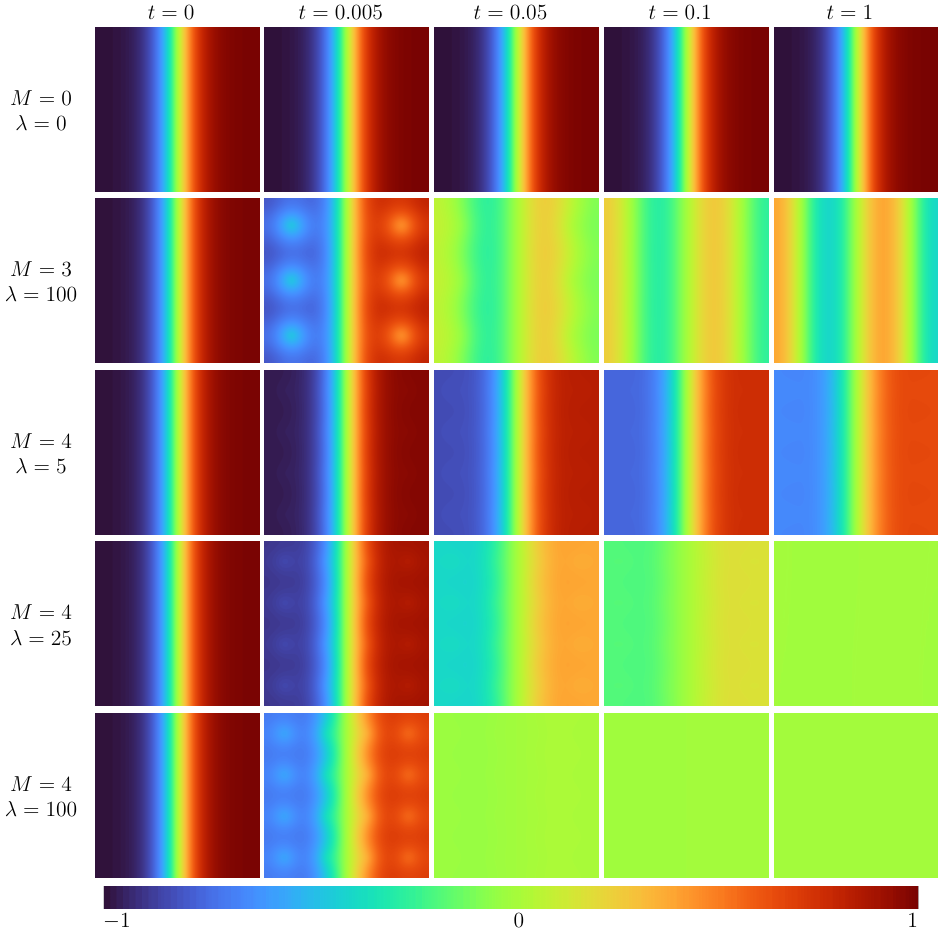}
 \caption{Snapshots of of solutions $y(t,x)$ to controlled system \eqref{eq:ch1}--\eqref{eq:ch2} with pointwise feedback $\mathcal{F}=\mathcal{F}_{M,\lambda}$ for time $t \in \{0,0.005,0.05,0.1,1\}$ (left to right) and different combinations of $M$ and $\lambda$ (top to bottom).
    \label{Fig:Simulation}}
\end{figure}
The first line depicts the evolution without feedback $(M = 0,\lambda = 0)$, confirming that $y_0$ is, in fact, very close to the stable equilibrium of the uncontrolled dynamics.
In the second line $(M=3,\lambda=100)$, the number of actuators $M^2=9$ is still too small, so that even a strong feedback with $\lambda=100$ is not sufficient to drive the system into the desired state. 
In the third line $(M=4,\lambda=5)$, the strength of the feedback $\lambda=5$ is too small, so a sufficiently high number of actuators $M^2=16$ is not sufficient.
The results of the fourth line $(M=4,\lambda=25)$ show the desired exponential convergence to the target solution, which here is an unstable equilibrium of \eqref{eq:ch1}--\eqref{eq:ch2}.
As depicted in the last line of Figure~\ref{Fig:Simulation}, the convergence can further accelerate by increasing the strength $\lambda=100$ of the feedback.
These observations are also in perfect agreement with the results depicted in Figure~\ref{fig:spectral}.

In Figure~\ref{fig:expdecay}, we show the evolution of distance $\log\|z(t)\|_{L^2}^2$ for different choices of feedback parameters $M$, $\lambda$ over time. 
\begin{figure}[ht!]
	\centering
\includegraphics[width=.99\textwidth,page=4]{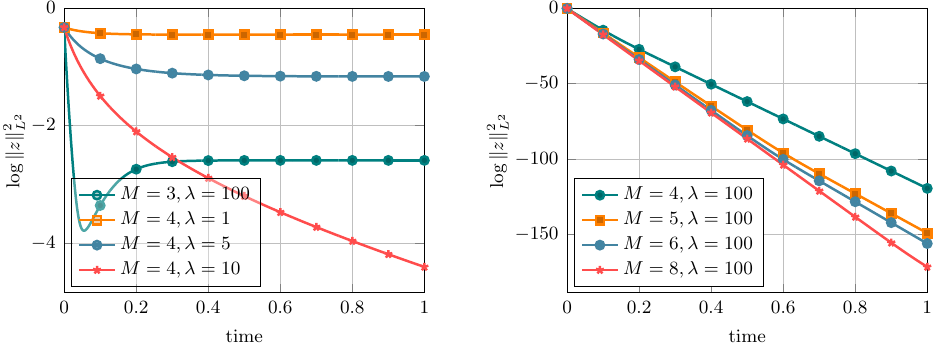}
    \caption{Evolution of distance $\log\|z(t)\|_{L^2}^2$ for various combinations of $M$ and $\lambda$; left plot: Non-convergence for too small $\lambda$ or $M$; right plot: exponential decay for $M$ and $\lambda$ sufficiently large.}
    \label{fig:expdecay}
\end{figure}
In the non-convergent cases (left plot), we observe stagnation of the pseudo-energies at some level, indicating a non-desired stable equilibrium of the controlled system. 
For sufficiently large values of $M$, $\lambda$ (right plot), on the other hand, we obtain the predicted exponential convergence, indicated by linear decay of the logarithmic distance measure.
As can be deduced from the results, 
sufficient large values of $M$ and $\lambda$ are necessary to obtain exponential stabilisation toward the target. Above the critical thresholds, the gain in further increasing the number of actuators is moderate, while further increasing the strength of the feedback results in faster decay. Again, these observations are again in perfect agreement with the spectral estimates investigated in Section~\ref{sec:num_spectral} and depicted in Figure~\ref{fig:spectral}.

\section*{Acknowledgements} 
This work was supported by the state of Upper Austria.
HE further acknolwedges supported by the Austrian Science Fund (FWF) via grant 10.55776/F90.


\end{document}